\documentclass{article}
\usepackage{amsmath}
\usepackage{mathtools}  
\usepackage{ifthen}
\usepackage{amssymb}
\usepackage{a4wide}
\usepackage{setspace}
\usepackage[final]{listings}
\usepackage{amsthm}
\usepackage{amsfonts}
\usepackage{enumitem,ulem}
\usepackage[cp1250]{inputenc}
\usepackage[T1]{fontenc}
\usepackage{booktabs}
\usepackage{graphicx}
\usepackage{mathptmx} 
\usepackage{float}
\usepackage{tabularx}
\usepackage{fancyhdr}
\usepackage{bigints}
\usepackage[square,numbers,sort]{natbib}
\usepackage[a4paper]{geometry}
\usepackage{tikz}
\usepackage{eurosym}
\usepackage[toc,page]{appendix}
\usepackage{pgfplots}
\usepackage{colonequals}
\usepackage{todonotes}
\usepackage{hyperref}
\usepackage{ dsfont }
\usepackage{bm}

\usetikzlibrary{plotmarks}

\newtheorem{theorem}{Theorem}

\newtheorem{definition}{Definition}

\newtheorem{lemma}[theorem]{Lemma}

\def\be{\begin{equation}}
\def\ee{\end{equation}}

\def\geq{\geqslant}

\def\N{\mathbb N}
\def\R{\mathbb R}

\def\grad{\text{grad}}
\def\Hess{\text{Hess}}

\def\b[#1]{{\textbf #1}}

\def\exp{\text{exp}}

\title{A direct Proof for Quadratic Convergence of the Geometric Newton Method
\thanks{%
The research was supported by
the Swiss National Fund under grant SNF 140635.
}
}

\author{Markus Sprecher \footnote{Seminar for Applied Mathematics, ETH Zürich, Rämistrasse 101, 8092 Zürich, Switzerland.}}

\begin{document}

\maketitle

\section*{Introduction}

We consider the problem of numerically computing a critical point of a functional $J\colon M\rightarrow \R$ where $M$ is a Riemannian manifold. Due to local quadratic convergence a popular choice to solve this problem is the geometric Newton method. The proofs for quadratic convergence either use computations in a chart (Theorem 6.3.2 and Section 6.3.1 in \cite{absil}) or require additional geometric quantities such as parallel translation \cite{smith}. In this short note we provide a direct proof for quadratic convergence.

The Riemannian gradient $\grad J(p)\in T_pM$ and Riemannian Hessian $\Hess J(p)\colon T_pM\rightarrow T_pM$ where $T_pM$ is the tangent space at $p\in M$ are usually introduced using covariant derivatives. By Proposition 5.5.5 of \cite{absil}, the Riemannian gradient and Riemannian Hessian can also be seen as the classical gradient and classical Hessian of the composition $J\circ R_p$ at $0$, where $R_p\colon T_pM\rightarrow M$ is a retraction of order $2$ as defined below.
\begin{definition}
Let $M$ be a Riemannian manifold and $p\in M$. A $C^\infty$-map $R_p\colon T_pM\rightarrow M$ is called a retraction of order $k\in \N$ if
$$
d(R_p(v),\exp_p(v))\lesssim |v|^{k+1},
$$
where $d\colon M\times M\rightarrow \R_{\geq 0}$ denotes the geodesic distance on $M$, $\exp_p\colon T_pM\rightarrow M$ the exponential map, $|\cdot|$ the norm induced by the inner product on $T_pM$ and $\lesssim$ indicates that there is a constant such that the left hand side can be estimated by a constant times the right hand side for $|v|$ small enough.
\end{definition}
Taylor expansion of $J\circ R_p$ at $0$ yields
\be\label{equ::Taylor}
J(R_p(v))=J(p)+\langle v,\grad J(p)\rangle+\frac{1}{2}\langle v,\Hess J(p)v\rangle+\mathcal{O}(|v|^3),
\ee
where $\langle \cdot,\cdot \rangle$ is the inner product on $T_pM$ and $\mathcal{O}(|v|^3)$ a term which can be estimated by a constant (depending on the third derivative of $J\circ R_p$ at $0$) times $|v|^3$ for $|v|$ small enough.

Equation \eqref{equ::Taylor} together with the fact that $\Hess J(p)$ is self-adjoint uniquely determines $\grad J(p)\in T_pM$ and $\Hess J(p)\colon T_pM\rightarrow T_pM$ and we can therefore also use \eqref{equ::Taylor} as a definition for $\grad J(p)$ and $\Hess J(p)$. 
\begin{definition}\label{def::Hess}
Let $M$ be a Riemannian manifold and $R_p$ a retraction of order $1$ at $p\in M$. Then the Riemannian gradient $\grad J(p)\in T_pM$ and Riemannian Hessian $\Hess J(p)\colon T_pM\rightarrow T_pM$ are the unique vector respectively unique self-adjoint operator such that \eqref{equ::Taylor} holds.
\end{definition}
Note that here we require the retraction only to be of order $1$. The Riemannian Hessian will in general depend on the choice of the retraction. However, if the retraction is of order $2$ the Riemannian Hessian will coincide with the Riemannian Hessian of \cite{absil}. Definition \ref{def::Hess} will allow us to give a direct proof for quadratic convergence. We can now define the geometric Newton method.
\begin{definition}
Let $M$ be a Riemannian manifold, $R_p$ a retraction of order $1$ at $p\in M$ and $\grad J(p)\in T_pM$ and $\Hess J(p)\colon T_pM\rightarrow T_pM$ as in Definition \ref{def::Hess}. Assume that $\Hess J(p)$ is invertible. Then the geometric Newton method at $p\in M$ is given by the iteration
\be\label{def::New}
\phi(p):=R_p\left(-(\Hess J(p))^{-1}\grad J(p)\right).
\ee
\end{definition}

\section*{Proof of quadratic convergence}

We will need two lemmas to prove local quadratic convergence of the geometric Newton method. First, we show an estimate involving the Riemannian gradient and Riemannian Hessian. It can be regarded as a weak version of the fact that the derivative of the gradient is the Hessian.
\begin{lemma}\label{lem::}
Let $M$ be a Riemannian manifold, $p^*$ a critical point of a three times differentiable functional $J\colon M\rightarrow \R$, $p\in M$, $R_p\colon T_pM\rightarrow M$ a retraction of order $1$, $\grad J(p)\in T_pM$ and $\Hess J(p)\colon T_pM\rightarrow T_pM$ as in Definition \ref{def::Hess} and $v\in T_pM$ such that $p^*=R_p(v)$. Then we have
\be\label{equ::est}
\left|\grad J(p)+\Hess J(p)v\right|\lesssim |v|^2.
\ee
The implicit constant depends only on the third derivative of $J\circ R_p$ at $0$.
\end{lemma}
\begin{proof}
Let $w:=\grad J(p)+\Hess J(p)v$. Replacing $v$ in \eqref{equ::Taylor} by $v+tw$ and using that $\Hess J(p)$ is self-adjoint yields
\begin{eqnarray*}
J(R_p(v+tw))&=&J(p)+\langle v+tw,\grad J(p)\rangle+\frac{1}{2}\langle v+tw,\Hess J(p)(v+tw)\rangle+\mathcal{O}(|v+tw|^3)\\
&=&J(p)+\langle v,\grad J(p)\rangle+\frac{1}{2}\langle v,\Hess J(p)v\rangle+\mathcal{O}(|v|^3)\\
&&+t\left(\langle w,\grad J(p)+\Hess J(p)v\rangle+\mathcal{O}(|w||v|^2)\right)+\mathcal{O}(t^2)\\
&=& J(p^*)+ t\left(|w|^2+\mathcal{O}(|w||v|^2)\right)+\mathcal{O}(t^2).
\end{eqnarray*}
Since $p^*$ is critical point of $J$ we have $J(R_p(v+tw))=J(p^*)+\mathcal{O}(t^2)$ hence $|w|^2=\mathcal{O}(|w||v|^2)$ and therefore $|w|\lesssim |v|^2$, which is equivalent to \eqref{equ::est}.
\end{proof}
Additionally, we will need the following Lemma.
\begin{lemma}\label{lem::2}
Let $M$ be a Riemannian manifold and $R_p$ a retraction of order $1$ at $p\in M$. Then we have
$$
d(R_p(v),R_p(w))=|v-w|+\mathcal{O}(|v|^2+|w|^2).
$$ 
\end{lemma}
\begin{proof}
By Nash embedding theorem \cite{nash} we can assume that $M\subset \R^n$ for some $n\in \N$. As $R_p$ is a retraction of order $1$ we have $R_p(v)=p+v+\mathcal{O}(|v|^2)$. Furthermore we have $d(p,q)=|p-q|+\mathcal{O}(|p-q|^2)$. Hence we have
$$
d(R_p(v),R_p(w))=|R_p(v)-R_p(w)|+\mathcal{O}(|R_p(v)-R_p(w)|^2)=|v-w|+\mathcal{O}(|v|^2+|w|^2). \qedhere
$$
\end{proof}
We can now prove quadratic convergence.
\begin{theorem}
Let $M$ be a Riemannian manifold, $p^*$ a critical point of a three times differentiable functional $J\colon M\rightarrow \R$, $p\in M$ and $\phi(p)$ as defined in \eqref{def::New}. Assume that $\Hess J(p^*)$ defined in Definition \ref{def::Hess} is invertible. Then we have
$$
d(\phi(p),p^*)\lesssim d^2(p,p^*)
$$
for all $p$ in a neighborhood of $p^*$. The implicit constant depends only on the third derivative of $J\circ R_{p^*}$ at $0$ and on the operator norm of $(\Hess J(p^*))^{-1}$.
\end{theorem}
\begin{proof}
Let $v$ be as in Lemma \ref{lem::}. By differentiability of $J$ and Lemma \ref{lem::} we have
\begin{eqnarray*}
\left|(\Hess J(p))^{-1}\grad J(p)+v\right|&\lesssim& \left\|(\Hess J(p))^{-1}\right\|_{T_pM\rightarrow T_pM} \left|\grad J(p)+\Hess J(p)v\right|\\
&\lesssim& \left(\left\|(\Hess J(p^*))^{-1}\right\|_{T_pM\rightarrow T_pM}+\mathcal{O}(d(p,p^*)) \right)|v|^2\\
&\lesssim& |v|^2.
\end{eqnarray*}
By Lemma \ref{lem::2} we now have
$$
d\left(\phi(p),p^*\right)= d\left(R_p(-(\Hess J(p))^{-1}\grad J(p)),R_p(v)\right)
\lesssim \left|(\Hess J(p))^{-1}\grad J(p)+v\right|
\lesssim |v|^2
\lesssim d^2(p,p^*). \qedhere
$$
\end{proof}

\bibliographystyle{plainnat}
\bibliography{riem_rev}

\end{document}